\documentclass[11pt,twoside]{article}
\usepackage{mathrsfs}
\usepackage{amsmath}
\usepackage{amsthm}
\input amssymb.sty
\usepackage{amsfonts}
\usepackage{amssymb}
\usepackage{latexsym}
\usepackage[all]{xy}

\date{\empty}
\allowdisplaybreaks[4] \footskip=15pt

\numberwithin{equation}{section} \theoremstyle{plain}
\newtheorem*{thm*}{Main Theorem}
\newtheorem{theorem}{Theorem}[section]
\newtheorem{corollary}[theorem]{Corollary}
\newtheorem*{corollary*}{Corollary}

\newtheorem*{claim*}{Claim}
\newtheorem{lemma}[theorem]{Lemma}
\newtheorem*{lemma*}{Lemma}

\newtheorem*{proposition*}{Proposition}
\newtheorem{remark}[theorem]{Remark}
\newtheorem*{remark*}{Remark}

\newtheorem*{example*}{Example}

\newtheorem*{question*}{Question}

\newtheorem*{definition*}{Definition}

\newtheorem*{acknowledgements*}{ACKNOWLEDGEMENTS}

\begin{document}
\begin{center}
{\large  \bf Additive and product properties of Drazin inverses of elements in a ring}\\
\vspace{0.8cm} {\small \bf  Huihui Zhu, \ \ Jianlong Chen\footnote{Corresponding author.
Department of Mathematics, Southeast University, Nanjing 210096, China.
Email: ahzhh08@sina.com (H. Zhu), jlchen@seu.edu.cn(J. Chen)}}
\end{center}

\bigskip

{ \bf  Abstract:}  \leftskip0truemm\rightskip0truemm  We study the Drazin inverses of the sum and product of two elements in a ring. For Drazin invertible elements $a$ and $b$ such that $a^2b=aba$ and $b^2a=bab$, it is shown that $ab$ is Drazin invertible and that $a+b$ is Drazin invertible if and only if $1+a^Db$ is Drazin invertible. Moreover, the formulae of $(ab)^D$ and $(a+b)^D$ are presented. Thus, a generalization of the main result of Zhuang, Chen et al. (Linear Multilinear Algebra 60 (2012) 903-910) is given.
\bigskip


\section { \bf Introduction}

~~~~Throughout this paper, $R$ is an associative ring with unity $1$. The symbols $R^D$, $R^{\rm nil}$ denote, the sets of Drazin invertible elements, nilpotent elements of $R$, respectively. The commutant of an element $a\in R$ is defined as ${\rm comm}(a)=\{x\in R: ax=xa\}$. An element $a\in R$ is Drazin invertible if there  exists $b\in R$ such that
$$b\in {\rm comm}(a),~ bab=b, ~a^k=a^{k+1}b \eqno(1.1)$$
for some positive integer $k$. The such least $k$ is called the Drazin index of $a$, denoted by $k={\rm ind}(a)$. If ${\rm ind}(a)=1$, then $b$ is called the group inverse of $a$ and denoted by $a^\#$.

The conditions in (1.1) are equivalent to
$$b\in {\rm comm}(a), ~bab=b, ~a-a^2b\in R^{\rm nil}.\eqno(1.2)$$
 Drazin [7] proved that if $a$ is Drazin invertible and $ab=ba$, then $a^Db=ba^D$. By $a^\pi=1-aa^D$ we mean the spectral idempotent of $a$.

The problem of Drazin inverse of the sum of two Drazin invertible elements was first considered by Drazin in his celebrated paper [7]. It was proved that $(a+b)^D=a^D+b^D$ under the condition that $ab=ba=0$ in associative rings. It is well known that the product $ab$ of two commutative Drazin invertible elements $a$, $b$ is Drazin invertible and $(ab)^D=a^Db^D=b^Da^D$ in a ring. In 2011, Wei and Deng [12] considered the relations between the Drazin inverses of $A+B$ and $1+A^DB$ for two commutative complex matrices $A$ and $B$. For two commutative Drazin invertible elements $a,b\in R$,Zhuang, Chen et al. [16] proved that $a+b$ is Drazin invertible if and only if $1+a^Db$ is Drazin invertible. Moreover, the representation of $(a+b)^D$ was obtained. More results on Drazin inverse can be found in [1-6,8,11-16].

For any elements $a, b\in R$, $ab=ba$ implies that $a^2b=aba$ and $b^2a=bab$. However, the converse need not be true. For example, take
\begin{center}
$a=\left(
     \begin{array}{cc}
       1 & 0 \\
       0 & 0\\
     \end{array}
   \right)
$, $b=\left(
        \begin{array}{cc}
          0 & 0  \\
          1 & 0 \\
          \end{array}
      \right)
$.
\end{center}
It is easy to get $a^2b=aba$ and $b^2a=bab$. However, $ab\neq ba$. Under the conditions $P^2Q=PQP$ and $Q^2P=QPQ$, Liu, Wu and Yu [9] characterized the relations between the Drazin inverses of $P+Q$ and $1+P^DQ$ for complex matrices $P$ and $Q$ by using the methods of splitting complex matrices into blocks. In this paper, we extend the results in [9] to a ring $R$. For $a, b\in R^D$, it is shown that $ab\in R^D$ and that $a+b\in R^D$ if and only if $1+a^Db\in R^D$ under the conditions $a^2b=aba$ and $b^2a=bab$. Moreover, the expressions of $(ab)^D$ and $(a+b)^D$ are presented. Consequently, some results in [9,12,16] can be deduced from our results.

\section{\bf Some lemmas}

In this section, we start with some useful lemmas.
\begin{lemma} Let $a,b\in R$ with $a^2b=aba$. Then for any positive integer $i$, the following hold:\\
$(1)$ $ a^{i+1}b =a^iba=aba^i,$ $a^{2i}b=a^iba^i,$\hfill$(2.1)$\\
$(2)$ $(ab)^i=a^ib^i.$\hfill$(2.2)$
\end{lemma}
\begin{proof} (1) Since $a^2b=aba$, we have $a^{i+1}b=a^{i-1}a^2b=a^{i-1}aba=a^iba$. This shows that $a^ib\in {\rm comm}(a)$ and $a^ib\in {\rm comm}(a^i)$.

From $ab\in {\rm comm}(a)$, it follows that $ab\in {\rm comm}(a^i)$, i.e., $aba^i=a^{i+1}b$. Thus, $a^{i+1}b =a^iba=aba^i$.

Therefore, we also obtain $a^{2i}b=a^iba^i$ by $a^ib\in {\rm comm}(a^i)$.

(2) follows by induction.
\end{proof}
\begin{lemma} Let $a,b\in R$ with $a^2b=aba$ and $b^2a=bab$.\\
$(1)$ If $a$ or $b$ is nilpotent, then $ab$ and $ba$ are nilpotent.\\
$(2)$ If $a$ and $b$ are nilpotent, then $a+b$ is nilpotent.
\end{lemma}
\begin{proof}  (1) By Lemma 2.1(2).

(2) Note that
$(a+b)^k=\displaystyle{\sum_{i=0}^{k-1}}C_{k-1}^i(a^{k-i}b^i+b^{k-i}a^i)$ in [9]. We have $(a+b)^k=0$ by taking $k={\rm ind}(a)+{\rm ind}(b)$.
\end{proof}

\begin{lemma}Let $a,b \in R$ with $a^2b=aba$ and $b^2a=bab$. If $a \in R^D$, then\\
$(1)$ $(a^D)^2b=a^Dba^D,$\hfill$(2.3)$\\
$(2)$ $b^2a^D=ba^Db.$\hfill$(2.4)$
\end{lemma}
\begin{proof} (1) Since $a^2b=aba$, we have $ab\in {\rm comm}(a^D)$ by [7, Theorem 1]. Hence,
\begin{center}
$(a^D)^2b=(a^D)^2a^Dab=(a^D)^2aba^D=a^Dba^D$.
\end{center}

(2) Note that $ab\in {\rm comm}(a^D)$. It follows that
\begin{center}
$b^2a^D=b^2a(a^D)^2=bab(a^D)^2=b(a^D)^2ab=ba^Db$.
\end{center}
\end{proof}

\begin{lemma} Let $a,b \in R^D$ with $a^2b=aba$ and $b^2a=bab$. Then\\
$(1)$ $\{ab, a^Db, ab^D, a^Db^D\} \subseteq {\rm comm}(a),$ \hfill$(2.5)$\\
$(2)$ $ \{ba,b^Da,ba^D,b^Da^D\} \subseteq {\rm comm}(b)$. \hfill$(2.6)$
\end{lemma}
\begin{proof}  It is enough to prove (1) since we can obtain (2) by the symmetry of $a$ and $b$.

(1) By hypothesis, we obtain $ab\in {\rm comm}(a)$ and
\begin{center}
$aa^Db=(a^D)^2a^2b=(a^D)^2aba=a^Dba$.
\end{center}

Since $ba\in {\rm comm}(b)$ implies that $ba\in {\rm comm}(b^D)$, it follows that
\begin{center}
$ab^Da=a(b^D)^2ba=aba(b^D)^2=a^2b(b^D)^2=a^2b^D$.
\end{center}

Note that $ab^Da=a^2b^D$. We get
\begin{center}
$aa^Db^D=(a^D)^2a^2b^D=(a^D)^2ab^Da=a^Db^Da$.
\end{center}
\end{proof}

\begin{lemma} Let $a, b\in R^D$ and $\xi=1+a^Db$. If $a^2b=aba$ and $b^2a=bab$, then
$\{a, a^D, ab, a^Db, ab^D, a^Db^D\} \subseteq {\rm comm}(\xi).\hfill(2.7)$
\end{lemma}
\begin{proof} Since $a^Db\in {\rm comm}(a)$, we have $a\in {\rm comm}(\xi)$. By [7, Theorem 1], $a^D\in {\rm comm}(\xi)$.

Observing that $(a^Db)ab \stackrel{(2.5)}{=}a(a^Db)b=a^D(ab)b=ab(a^Db)$, it follows that $ab\in {\rm comm}(\xi)$.

Similarly, $a^Db, ab^D, a^Db^D\in {\rm comm}(\xi)$.

Hence, $\{a, a^D, ab, a^Db, ab^D, a^Db^D\} \subseteq {\rm comm}(\xi)$.
\end{proof}
\begin{lemma} Let $a, b\in R^D$ with $a^2b=aba$ and $b^2a=bab$. Then for any positive integer $i$, the following hold:\\
$(1)$ $ab^Db^Da=(ab^D)^2=a^2(b^D)^2$, \hfill$(2.8)$\\
$(2)$ $(ab^D)^{i+1}=ab^D(b^Da)^i=a^{i+1}(b^D)^{i+1}$, \hfill$(2.9)$\\
$(3)$ $(a^Db)^{i+1}=a^Db(ba^D)^i=(a^D)^{i+1}b^{i+1}$,\hfill$(2.10)$\\
$(4)$ $ba^Da^Db=(ba^D)^2=b^2(a^D)^2$, \hfill$(2.11)$\\
$(5)$ $(ba^D)^{i+1}=ba^D(a^Db)^i=b^{i+1}(a^D)^{i+1}$, \hfill$(2.12)$\\
$(6)$ $(b^Da)^{i+1}=b^Da(ab^D)^i=(b^D)^{i+1}a^{i+1}$.\hfill$(2.13)$
\end{lemma}
\begin{proof} It is sufficient to prove (1)-(3).

(1) We have
$$ab^Db^Da=ab^D(b^Da)\stackrel{(2.6)}{=}a(b^Da)b^D=(ab^D)^2.$$

According to Lemma 2.4, we get
$$ab^Db^Da=ab^D(ab^D)\stackrel{(2.5)}{=}a(ab^D)b^D=a^2(b^D)^2.$$

(2) It is just (1) for $i=1$. Assume that the equality holds for $i=k$, i.e., $$(ab^D)^{k+1}=ab^D(b^Da)^k=a^{k+1}(b^D)^{k+1}.$$
For $i=k+1$,
\begin{eqnarray*}
(ab^D)^{k+2}=ab^D(ab^D)^{k+1}=ab^Dab^D(b^Da)^k \stackrel{(2.8)}{=}ab^Db^Da (b^Da)^k=ab^D(b^Da)^{k+1}
\end{eqnarray*}
and

$(ab^D)^{k+2}=ab^D(ab^D)^{k+1}=ab^Da^{k+1}(b^D)^{k+1}\stackrel{(2.5)}{=}a^{k+1}ab^D(b^D)^{k+1}= a^{k+2}(b^D)^{k+2}$.

Thus, (2) holds for any positive integer $i$.

(3) Its proof is similar to (2).
\end{proof}

\begin{lemma} Let $a, b\in R^D$ with $a^2b=aba$ and $b^2a=bab$. If $a_1=b^\pi a^\pi b$ and $a_2=bb^Daa^\pi$, then $a_1-a_2$ is nilpotent.
\end{lemma}

\begin{proof} Firstly, we prove $a_1=b^\pi a^\pi b$ is nilpotent. According to Lemma 2.6, we have the following equalities
 $$b^\pi ba^\pi b=b^2b^\pi a^\pi \eqno(2.14)$$
 and
$$ b a^\pi b^\pi =bb^\pi a^\pi.\eqno(2.15)$$

Hence, we obtain
\begin{eqnarray*}
(b^\pi a^\pi b)^3&=& b^\pi a^\pi( b b^\pi a^\pi b )b^\pi a^\pi b \stackrel{(2.14)}{=} b^\pi a^\pi(  b^\pi b^2 a^\pi )b^\pi a^\pi b \\
&=& b^\pi a^\pi b^\pi (b^2 a^\pi b^\pi) a^\pi b \stackrel{(2.15)}{=}b^\pi a^\pi b^\pi (b^2 b^\pi a^\pi ) a^\pi b \\
&=& b^\pi a^\pi (b b^\pi)^2 a^\pi  a^\pi b= b^\pi a^\pi (b b^\pi)^2 a^\pi b.
\end{eqnarray*}

By induction, $(b^\pi a^\pi b)^{n+1}=b^\pi a^\pi (bb^\pi)^n a^\pi b$. Since $bb^\pi$ is nilpotent, $b^\pi a^\pi b=a_1$ is nilpotent.

Secondly, we show that $a_2=bb^Daa^\pi$ is nilpotent.
As
\begin{eqnarray*}
(bb^D aa^\pi)^2 &=& (bb^D aa^\pi)(bb^D aa^\pi)= bb^Da(1-aa^D)bb^Da(1-aa^D)\\
&=& bb^D(1-aa^D)abb^Da(1-aa^D)\\
&\stackrel{(2.5)}{=}& bb^Dab(1-aa^D)b^Da(1-aa^D)\\
&\stackrel{(2.6)}{=}& bbb^Da(1-aa^D)b^Da(1-aa^D)\\
&=& bbb^D(1-aa^D)ab^Da(1-aa^D)\\
&\stackrel{(2.5)}{=}& bbb^Dab^D(1-aa^D)a(1-aa^D)\\
&\stackrel{(2.6)}{=}& bbb^Db^Da(1-aa^D)a(1-aa^D)\\
&=& bb^Da(1-aa^D)a(1-aa^D)\\
&=& bb^D (aa^\pi)^2,
\end{eqnarray*}
by induction, $(bb^D aa^\pi)^n=bb^D (aa^\pi)^n$. Since $aa^\pi$ is nilpotent, $bb^D aa^\pi=a_2$ is nilpotent.

Finally, we prove that $a_1-a_2$ is nilpotent.

Since
\begin{center}
 $a_1^2a_2= b^\pi a^\pi b b^\pi a^\pi b bb^Daa^\pi\stackrel{(2.6)}{=} b^\pi a^\pi bb^\pi b^2b^Daa^\pi=0$
\end{center}
and
\begin{eqnarray*}
a_2a_1&=& bb^Daa^\pi b^\pi a^\pi b=bb^Da^\pi (a b^\pi) a^\pi b\\
&\stackrel{(2.5)}{=}& bb^D(a b^\pi) a^\pi  a^\pi b=bb^D(a b^\pi) a^\pi b\\
&=& b(b^Da) b^\pi a^\pi b\stackrel{(2.6)}{=}b b^\pi (b^Da) a^\pi b\\
&=&0,
\end{eqnarray*}

we have $a_1^2a_2=a_1a_2a_1=0$ and $a_2^2a_1=a_2a_1a_2=0$.

As $a_1$ and $a_2$ are nilpotent, $b^\pi a^\pi b-bb^Daa^\pi=a_1-a_2$ is nilpotent by Lemma 2.2(2).
\end{proof}

\begin{lemma} Let $a, b\in R^D$ with $a^2b=aba$ and $b^2a=bab$ and $\xi=1+a^Db\in R^D$. Suppose $b_1=a\xi\xi^\pi+\xi^Daa^\pi$ and $b_2=b^\pi a^\pi b-bb^Daa^\pi$. Then $b_1+b_2$ is nilpotent.
\end{lemma}
\begin{proof} Since $a\xi\xi^\pi$, $\xi^Daa^\pi$ are nilpotent and $a\xi \xi^\pi$ commutes with $\xi^Daa^\pi$, it follows that $b_1=a\xi\xi^\pi+\xi^Daa^\pi$ is nilpotent. By Lemma 2.7, $a_1-a_2=b_2$ is nilpotent.

First we give two useful equalities
$$aa^\pi bb^\pi a=a^\pi ab^\pi (ba) \stackrel{(2.6)}{=}a^\pi (ab)ab^\pi \stackrel{(2.5)}{=}a^2a^\pi bb^\pi\eqno(2.16)$$
and
$$bab^\pi a^\pi b a \stackrel{(2.5)}{=}ba^\pi ab^\pi ba=ba^\pi (abb^\pi)a\stackrel{(2.5)}{=}b(abb^\pi)a^\pi a=(ba)bb^\pi aa^\pi\stackrel{(2.6)}{=}b^2b^\pi a^2a^\pi. \eqno(2.17)$$

According to Lemma 2.5, we have
\begin{eqnarray*}
b_1^2b_2&=&(a\xi \xi^\pi+\xi^Daa^\pi)^2(b^\pi a^\pi b -bb^Daa^\pi)\\
&\stackrel{(2.7)}{=}& (a^2\xi^2\xi^\pi+(\xi^D)^2a^2a^\pi)(b^\pi a^\pi b-bb^Daa^\pi)\\
&=& a^2\xi^2\xi^\pi b^\pi a^\pi b -a^2\xi^2\xi^\pi bb^Daa^\pi+(\xi^D)^2a^2a^\pi b^\pi a^\pi b-(\xi^D)^2a^2a^\pi bb^Daa^\pi\\
&\stackrel{(2.7)}{=}& a^2a^\pi\xi^2\xi^\pi bb^\pi -a^3a^\pi\xi^2\xi^\pi bb^D+(\xi^D)^2a^2a^\pi bb^\pi-(\xi^D)^2a^3a^\pi bb^D.
\end{eqnarray*}

By Lemma 2.4 and Lemma 2.5, we obtain
\begin{eqnarray*}
&&b_1b_2b_1=(a\xi\xi^\pi+\xi^Daa^\pi)(b^\pi a^\pi b -bb^Daa^\pi)(a\xi\xi^\pi+\xi^Daa^\pi)\\
&=&(a\xi\xi^\pi b^\pi a^\pi b -a\xi\xi^\pi bb^Daa^\pi +\xi^Daa^\pi b^\pi a^\pi b-\xi^Da^2a^\pi bb^D)(a\xi\xi^\pi+\xi^Daa^\pi)\\
&=& a\xi\xi^\pi b^\pi a^\pi b a\xi\xi^\pi-a\xi\xi^\pi bb^Daa^\pi a\xi\xi^\pi+\xi^Daa^\pi b^\pi a^\pi b a\xi\xi^\pi-\xi^Da^2a^\pi bb^Da\xi\xi^\pi \\
&&+a\xi\xi^\pi b^\pi a^\pi b \xi^Daa^\pi -a\xi\xi^\pi bb^Daa^\pi \xi^Daa^\pi+\xi^Daa^\pi b^\pi a^\pi b\xi^Daa^\pi- \xi^Da^2a^\pi bb^D\xi^D aa^\pi \\
&\stackrel{(2.7)}{=}& a\xi\xi^\pi a\xi\xi^\pi b^\pi a^\pi b -a\xi\xi^\pi aa^\pi a\xi\xi^\pi bb^D+\xi^Daa^\pi a^\pi a\xi\xi^\pi b^\pi b -\xi^Da^2a^\pi a\xi\xi^\pi bb^D \\
&&+a\xi\xi^\pi a^\pi \xi^Daa^\pi b^\pi b -a\xi\xi^\pi aa^\pi \xi^Daa^\pi bb^D+\xi^Daa^\pi a^\pi \xi^Daa^\pi b^\pi b- \xi^Da^2a^\pi \xi^D aa^\pi bb^D \\
&\stackrel{(2.16)}{=}& a^2a^\pi\xi^2\xi^\pi bb^\pi -a^3a^\pi\xi^2\xi^\pi bb^D+(\xi^D)^2a^2a^\pi bb^\pi-(\xi^D)^2a^3a^\pi bb^D.
\end{eqnarray*}

Hence, $b_1^2b_2=b_1b_2b_1$.

In view of Lemma 2.7 and $b^Db^\pi=0$, we have
\begin{eqnarray*}
b_2^2b_1&=&(b^\pi a^\pi b -bb^Daa^\pi)^2(a\xi \xi^\pi +\xi^Daa^\pi)\\
&\stackrel{(2.6)}{=}&(b^\pi a^\pi b^2b^\pi a^\pi-b^\pi a^\pi b^2b^Daa^\pi +bb^Da^2a^\pi)(a\xi \xi^\pi +\xi^Daa^\pi)\\
&=&b^\pi a^\pi b^2b^\pi aa^\pi\xi\xi^\pi-b^\pi a^\pi b^2b^Daa^\pi a\xi\xi^\pi+b^2b^Da^2a^\pi a\xi\xi^\pi\\
&&+b^\pi a^\pi b^2b^\pi a^\pi\xi^Daa^\pi-b^\pi a^\pi b^2b^Daa^\pi\xi^Daa^\pi+bb^Da^2a^\pi\xi^Daa^\pi\\
&\stackrel{(2.7)}{=}& b^\pi a^\pi b^2b^\pi aa^\pi\xi\xi^\pi-b^\pi a^\pi b^2b^Da^2a^\pi\xi\xi^\pi+b^2b^Da^3a^\pi \xi\xi^\pi\\
&&+b^\pi a^\pi b^2b^\pi aa^\pi\xi^D-b^\pi a^\pi b^2b^Da^2a^\pi\xi^D+bb^Da^3a^\pi\xi^D,
\end{eqnarray*}
and
\begin{eqnarray*}
b_2b_1b_2&=&(b^\pi a^\pi b -bb^Daa^\pi)(a\xi \xi^\pi +\xi^Daa^\pi)(b^\pi a^\pi b -bb^Daa^\pi)\\
&\stackrel{(2.7)}{=}&(b^\pi a^\pi b a\xi \xi^\pi+b^\pi a^\pi b aa^\pi\xi^D -bb^Da^2a^\pi \xi \xi^\pi-bb^Da^2a^\pi\xi^D)(b^\pi a^\pi b -bb^Daa^\pi)\\
&=&b^\pi a^\pi b a\xi \xi^\pi b^\pi a^\pi b +b^\pi a^\pi b aa^\pi\xi^D b^\pi a^\pi b -bb^Da^2a^\pi \xi \xi^\pi b^\pi a^\pi b\\
 &&-bb^Da^2a^\pi\xi^D b^\pi a^\pi b-b^\pi a^\pi b a\xi \xi^\pi bb^Daa^\pi -b^\pi a^\pi ba  a^\pi\xi^D bb^Daa^\pi \\
 &&+bb^Da^2a^\pi \xi\xi^\pi bb^Daa^\pi +bb^Da^2a^\pi \xi^D bb^Daa^\pi\\
&\stackrel{(2.6)}{=}&b^\pi a^\pi b a\xi \xi^\pi b^\pi a^\pi b +b^\pi a^\pi b aa^\pi\xi^D b^\pi a^\pi b -bb^Db^\pi a^2a^\pi \xi \xi^\pi  a^\pi b \\
&&-bb^D b^\pi a^2a^\pi\xi^D a^\pi b-b^\pi a^\pi b a\xi \xi^\pi bb^Daa^\pi -b^\pi a^\pi ba  a^\pi\xi^D bb^Daa^\pi \\
&&+bb^Da^2a^\pi \xi\xi^\pi bb^Daa^\pi +bb^Da^2a^\pi \xi^D bb^Daa^\pi\\
&=& b^\pi a^\pi b a\xi \xi^\pi b^\pi a^\pi b +b^\pi a^\pi b aa^\pi\xi^D b^\pi a^\pi b -b^\pi a^\pi b a\xi \xi^\pi bb^Daa^\pi \\
&&-b^\pi a^\pi ba  a^\pi\xi^D bb^Daa^\pi +bb^Da^2a^\pi \xi\xi^\pi bb^Daa^\pi +bb^Da^2a^\pi \xi^D bb^Daa^\pi\\
&\stackrel{(2.17)}{=}&  b^\pi a^\pi b^2b^\pi aa^\pi\xi\xi^\pi-b^\pi a^\pi b^2b^Da^2a^\pi\xi\xi^\pi+b^2b^Da^3a^\pi \xi\xi^\pi\\
&&+b^\pi a^\pi b^2b^\pi aa^\pi\xi^D-b^\pi a^\pi b^2b^Da^2a^\pi\xi^D+bb^Da^3a^\pi\xi^D.
\end{eqnarray*}

Therefore, $b_2^2b_1=b_2b_1b_2$.

By Lemma 2.2(2), it follows that $b_1+b_2$ is nilpotent.
\end{proof}
\section{  \bf Main results }
~~~In this section, we consider the formulae on the Drazin inverses of the product and sum of two elements of $R$.
\begin{theorem} Let $a,b \in R^D$ with $a^2b=aba$ and $b^2a=bab$. Then $ab\in R^D$ and
\begin{center}
$(ab)^D=a^Db^D$.
\end{center}
\end{theorem}
\begin{proof} Let $x=a^Db^D$. We prove that $x$ is the Drazin inverse of $ab$ by showing the following results:
(1) $(ab)x=x(ab)$; (2) $ x(ab)x=x$; (3) $(ab)^k=(ab)^{k+1}x$ for some positive integer $k$.

(1) Note that $a(ab)=(ab)a$ implies $(ab)a^D=a^D(ab)$. It follows that
\begin{eqnarray*}
 (ab)x &=& (ab)a^Db^D\stackrel{(2.5)}{=}a^Dabb^D \\
   &=& a(a^Db^D)b\stackrel{(2.5)}{=} a^Db^Dab \\
   &=& x(ab).
\end{eqnarray*}

$(2)$ We calculate directly that
\begin{eqnarray*}
  x(ab)x &=& (a^Db^D)aba^Db^D\stackrel{(2.5)}{=}a(a^Db^D)ba^Db^D=aa^Db(b^Da^D)b^D \\
   &=& aa^D(b^Da^D)bb^D=a(a^Db^D)a^Dbb^D \stackrel{(2.5)}{=}aa^D(a^Db^D)bb^D\\
   &=& x.
\end{eqnarray*}

$(3)$ Take $k={\rm max}\{{\rm ind}(a),{\rm ind}(b)\}$. Then
\begin{eqnarray*}
  (ab)^{k+1}x&=& (ab)^{k+1}a^Db^D=a^D(ab)^{k+1}b^D \\
   &\stackrel{(2.2)}{=}& a^Da^{k+1}b^{k+1}b^D=a^kb^k\\
   &\stackrel{(2.2)}{=}& (ab)^k.
\end{eqnarray*}

Hence, $(ab)^D=a^Db^D$.
\end{proof}
\begin{remark}
{\rm In [16, Lemma 2], for $a,b \in R^D$ with $ab=ba$, it is proved that $(ab)^D=a^Db^D=b^Da^D$. However, in Theorem 3.1, $(ab)^D=a^Db^D \neq b^Da^D$. Such as, take
$a=\left(
          \begin{array}{cc}
            1 & 0 \\
            1 & 0 \\
          \end{array}
        \right)
$ and $b=\left(
           \begin{array}{cc}
             1 & 0 \\
             0 & 0 \\
           \end{array}
         \right)
$. By calculations, we obtain $a^2b=aba$ and $b^2=bab$, but $a^Db^D \neq b^Da^D$}.
\end{remark}

\begin{theorem} Let $a,b \in R^D$. If $a^2b=aba$, $b^2a=bab$ and ${\rm ind}(a)=s$, then $a+b$ is Drazin invertible if and only if  $1+a^Db$ is Drazin invertible. Moreover, we have
\begin{eqnarray*}
(a+b)^D &=& a^D(1+a^Db)^D+a^\pi b[a^D(1+a^Db)^D]^2+\sum _{i=0}^{s-1}(b^D)^{i+1}(-a)^ia^\pi \\
 && +b^\pi a \sum_{i=0}^{s-2}(i+1)(b^D)^{i+2}(-a)^ia^\pi,
\end{eqnarray*}
and $(1+a^Db)^D=a^\pi +a^2a^D(a+b)^D$.
\end{theorem}

\begin{proof} Suppose that $a+b$ is Drazin invertible. We prove that $1+a^Db$ is Drazin invertible. Write $1+a^Db=a_1+b_1$, where $a_1=a^\pi$, $b_1=a^D(a+b)$.

Note that $(a^D)^2(a+b)=a^D(a+b)a^D$ and $(a+b)^2a^D=(a+b)a^D(a+b)$. By Theorem 3.1, it follows that $a^D(a+b)=b_1$ is Drazin invertible and
\begin{center}
$(b_1)^D=[a^D(a+b)]^D=(a^D)^D(a+b)^D=a^2a^D(a+b)^D$.
\end{center}

 From Lemma 2.4, we obtain that $a^Db$ commutes with $aa^D$. Hence, $a^D(a+b)\in {\rm comm}(a^\pi)$ and $a_1b_1=b_1a_1=0$. By [7, Corollary 1], it follows that $(1+a^Db)^D=a^\pi +a^2a^D(a+b)^D$.

Conversely, let $\xi=1+a^Db$ be Drazin invertible and
\begin{eqnarray*}
x &=& a^D\xi^D+a^\pi b(a^D\xi^D)^2+\sum _{i=0}^{s-1}(b^D)^{i+1}(-a)^ia^\pi+b^\pi a \sum_{i=0}^{s-2}(i+1)(b^D)^{i+2}(-a)^ia^\pi\\
&=& x_1+x_2,
\end{eqnarray*}
where $x_1= a^D\xi^D+a^\pi b(a^D\xi^D)^2$, $x_2=\displaystyle{\sum _{i=0}^{s-1}(b^D)^{i+1}(-a)^ia^\pi +b^\pi a \sum_{i=0}^{s-2}(i+1)(b^D)^{i+2}(-a)^ia^\pi}$.

According to $a^D\in {\rm comm}(\xi^D)$ and $(ba^D)^i=b^i(a^D)^i$, we have
\begin{eqnarray*}
(a+b)a^\pi b (a^D)^2&=&(a+b)(1-aa^D)b(a^D)^2=(a-a^2a^D+b-baa^D)b(a^D)^2\\
&=& ab(a^D)^2-a^2a^Db(a^D)^2+b^2(a^D)^2-baa^Db(a^D)^2\\
&=& (a^D)^2ab-a^2(a^D)^2a^Db+b^2(a^D)^2-ba(a^D)^2a^Db\\
&=& b^2(a^D)^2-b(a^D)^2b\\
&\stackrel{(2.6)}{=}& b^2(a^D)^2-(ba^D)^2\\
&\stackrel{(2.11)}{=}&0.
\end{eqnarray*}

Thus, $(a+b)a^\pi b(a^D\xi^D)^2=0$.

Similarly, $(a+b)b^\pi a(b^D)^2=0$. Hence, $(a+b)b^\pi a \displaystyle{\sum_{i=0}^{s-2}(i+1)}(b^D)^{i+2}(-a)^ia^\pi =0$.

Next, we show that $x$ is the Drazin inverse of $(a+b)$ in 3 steps.

Step 1. First we prove that $x(a+b)=(a+b)x$. Put $y_1=(a+b)a^D\xi^D$ and $y_2=(a+b)\displaystyle{\sum_{i=0}^{s-1}(b^D)^{i+1}(-a)^ia^\pi}$.
Then we have
\begin{eqnarray*}
(a+b)x &=&(a+b) [a^D \xi^D+a^\pi b(a^D \xi^D)^2+\sum _{i=0}^{s-1}(b^D)^{i+1}(-a)^ia^\pi \\
 && +b^\pi a \sum_{i=0}^{s-2}(i+1)(b^D)^{i+2}(-a)^ia^\pi]\\
 &=&(a+b)[a^D \xi^D + \sum _{i=0}^{s-1}(b^D)^{i+1}(-a)^ia^\pi]\\
 &=& y_1+y_2.
\end{eqnarray*}

Second we show that $x_1(a+b)=y_1$ and $x_2(a+b)=y_2$.

Since $a^D\xi^D=\xi^Da^D$, we get
\begin{eqnarray*}
x_1(a+b) &=& [a^D\xi^D+a^\pi b(a^D(\xi^D)^2](a+b)\\
&\stackrel{(2.5)}{=}& a^D(a+b)\xi^D+ a^\pi b(a^D)^2(a+b)(\xi^D)^2\\
&=& a^D(a+b)\xi^D+(a^\pi ba^D+a^\pi ba^Da^Db)(\xi^D)^2\\
&=& a^D(a+b)\xi^D+a^\pi ba^D\xi(\xi^D)^2\\
&=& a^D(a+b)\xi^D +a^\pi ba^D\xi^D\\
&=& (a^Da+a^Db+ba^D-aa^Dba^D)\xi^D\\
&=& (a+b)a^D \xi^D\\
&=& y_1.
\end{eqnarray*}

By induction,
\begin{eqnarray}
bb^D(ab^D)^i&=&(bb^Dab^D)^i=(b^Da)^i.
\end{eqnarray}

Note that $a^sa^\pi=0$ and
\begin{eqnarray}
\nonumber
b^Da^\pi b &=& b^D(1-aa^D)b=bb^D-(b^Da^D)ba\stackrel{(2.6)}{=}bb^D-bb^Da^Da  \\
&=& bb^Da^\pi.
\end{eqnarray}

We have
\begin{eqnarray*}
&&x_2(a+b)-y_2=[\sum _{i=0}^{s-1}(b^D)^{i+1}(-a)^ia^\pi +b^\pi a \sum_{i=0}^{s-2}(i+1)(b^D)^{i+2}(-a)^ia^\pi](a+b)\\
&&-(a+b)\sum_{i=0}^{s-1}(b^D)^{i+1}(-a)^ia^\pi\\
&=&-\sum _{i=0}^{s-1}(b^D)^{i+1}(-a)^{i+1}a^\pi+\sum_{i=0}^{s-1}(b^D)^{i+1}(-a)^ia^\pi b-b^\pi a\sum_{i=0}^{s-2}(i+1)(b^D)^{i+2}(-a)^{i+1}a^\pi\\
&&+b^\pi a\sum_{i=0}^{s-2}(i+1)(b^D)^{i+2}(-a)^ia^\pi b-(a+b)\sum_{i=0}^{s-1}(b^D)^{i+1}(-a)^ia^\pi\\
&=&-\sum_{i=0}^{s-1}(-b^Da)^{i+1}a^\pi +\sum_{i=0}^{s-1}b^D(-b^Da)^ia^\pi b-b^\pi ab^D\sum_{i=0}^{s-2}(i+1)(-b^Da)^ib^Da^\pi b\\
&&+b^\pi a(b^D)^2\sum_{i=0}^{s-2}(i+1)(-b^Da)^ia^\pi b-a\sum_{i=0}^{s-1}(b^D)^{i+1}(-a)^ia^\pi-b\sum_{i=0}^{s-1}(b^D)^{i+1}(-a)^ia^\pi\\
&\stackrel{(3.2)}{=}&-\sum_{i=0}^{s-1}(-b^Da)^{i+1}a^\pi+\sum_{i=0}^{s-1}(-b^Da)^ibb^Da^\pi+b^\pi \sum_{i=0}^{s-2}(i+1)(-ab^D)^{i+2}a^\pi\\
&& -b^\pi\sum_{i=0}^{s-2}(i+1)(-ab^D)^{i+1}bb^Da^\pi +\sum_{i=0}^{s-1}(-ab^D)^{i+1}a^\pi -\sum_{i=0}^{s-1}(-b^Da)^ibb^Da^\pi\\
&=& -\sum_{i=0}^{s-1}(-b^Da)^{i+1}a^\pi +\sum_{i=0}^{s-1}(-ab^D)^{i+1}a^\pi +b^\pi[\sum_{i=0}^{s-2}(i+1)(-ab^D)^{i+2}a^\pi\\
&&-\sum_{i=0}^{s-2}(i+1)(-ab^D)^{i+1}a^\pi]\\
&=& -\sum_{i=0}^{s-1}(-b^Da)^{i+1}a^\pi +\sum_{i=0}^{s-1}(-ab^D)^{i+1}a^\pi -b^\pi \sum_{i=1}^{s-1}(-ab^D)^ia^\pi\\
&=& -\sum_{i=0}^{s-1}(-b^Da)^{i+1}a^\pi +bb^D\sum_{i=1}^{s-1}(-ab^D)^ia^\pi\\
&\stackrel{(3.1)}{=}&  -\sum_{i=0}^{s-1}(-b^Da)^{i+1}a^\pi +\sum_{i=1}^{s-1}(-bb^Dab^D)^ia^\pi\\
&=& -\sum_{i=0}^{s-1}(-b^Da)^{i+1}a^\pi +\sum_{i=1}^{s-1}(-b^Da)^ia^\pi\\
&=& -\sum_{i=1}^s(-b^Da)^ia^\pi +\sum_{i=1}^{s-1}(-b^Da)^ia^\pi\\
&=& -\sum_{i=1}^{s-1}(-b^Da)^ia^\pi -(-b^Da)^sa^\pi +\sum_{i=1}^{s-1}(-b^Da)^ia^\pi\\
&\stackrel{(2.13)}{=}&-\sum_{i=1}^{s-1}(-b^Da)^ia^\pi -(-b^D)^sa^sa^\pi +\sum_{i=1}^{s-1}(-b^Da)^ia^\pi\\
&=&0.
\end{eqnarray*}

Hence, $x_2(a+b)=y_2$. It follows that $x(a+b)=(a+b)x$.

Step 2. We show that $x(a+b)x=x$. Note that $a+b=a\xi+a^\pi b$. We have
\begin{eqnarray*}
&&x(a+b)x = x(a+b)[a^D\xi^D+\sum_{i=0}^{s-1}(b^D)^{i+1}(-a)^ia^\pi]\\
&=& (a+b)[a^D\xi^D+\sum_{i=0}^{s-1}(b^D)^{i+1}(-a)^ia^\pi][a^D\xi^D+\sum_{i=0}^{s-1}(b^D)^{i+1}(-a)^ia^\pi]\\
&=&(a+b)(a^D\xi^D)^2+(a+b)a^D\xi^D\sum_{i=0}^{s-1}(b^D)^{i+1}(-a)^ia^\pi\\
&&+(a+b)\sum_{i=0}^{s-1}(b^D)^{i+1}(-a)^ia^\pi \sum_{i=0}^{s-1}(b^D)^{i+1}(-a)^ia^\pi\\
&=&z_1+z_2+z_3,
\end{eqnarray*}
where $z_1=(a+b)(a^D\xi^D)^2$, $z_2=(a+b)a^D\xi^D\displaystyle{\sum_{i=0}^{s-1}(b^D)^{i+1}(-a)^ia^\pi}$
 and $$z_3=(a+b)\displaystyle{\sum_{i=0}^{s-1}(b^D)^{i+1}(-a)^ia^\pi \sum_{i=0}^{s-1}(b^D)^{i+1}(-a)^ia^\pi}.$$

Now we prove that $z_1+z_2+z_3=x$.

Firstly, we have
\begin{eqnarray*}
z_1&=& (a+b)(a^D\xi^D)^2=(a\xi+a^\pi b)(a^D\xi^D)^2\\
 &=& a\xi(a^D\xi^D)^2+a^\pi b(a^D\xi^D)^2\\
&\stackrel{(2.5)}{=}&a^D\xi^D+ a^\pi b(a^D\xi^D)^2,
\end{eqnarray*}
\begin{eqnarray*}
z_2&=& (a+b)a^D\xi^D\sum_{i=0}^{s-1}(b^D)^{i+1}(-a)^ia^\pi\\
 &=& \xi^D aa^D\sum_{i=0}^{s-1}(b^D)^{i+1}(-a)^ia^\pi +b\xi^Da^D \sum_{i=0}^{s-1}(b^D)^{i+1}(-a)^ia^\pi\\
&=&\xi^D \sum_{i=0}^{s-1}a^Da b^D(-b^Da)^ia^\pi +b\xi^D \sum_{i=0}^{s-1}(a^D)^2ab^D(-b^Da)^ia^\pi\\
&\stackrel{(2.9)}{=}&-\xi^D\sum_{i=0}^{s-1}a^D(-ab^D)^{i+1}a^\pi -b\xi^D\sum_{i=0}^{s-1}(a^D)^2(-ab^D)^{i+1}a^\pi\\
&\stackrel{(2.5)}{=}& -\xi^D\sum_{i=0}^{s-1}(-ab^D)^{i+1}a^Da^\pi -b\xi^D\sum_{i=0}^{s-1}(-ab^D)^{i+1}(a^D)^2a^\pi\\
&=& 0.
\end{eqnarray*}

Secondly, we show that
\begin{eqnarray}
z_3&=& \sum _{i=0}^{s-1}(b^D)^{i+1}(-a)^ia^\pi+b^\pi a \sum_{i=0}^{s-2}(i+1)(b^D)^{i+2}(-a)^ia^\pi.
\end{eqnarray}

Indeed, we have
\begin{eqnarray*}
&&z_3= (a+b)\sum_{i=0}^{s-1}(b^D)^{i+1}(-a)^ia^\pi \sum_{i=0}^{s-1}(b^D)^{i+1}(-a)^ia^\pi\\
&=& b\sum_{i=0}^{s-1}(b^D)^{i+1}(-a)^ia^\pi \sum_{i=0}^{s-1}(b^D)^{i+1}(-a)^ia^\pi +a\sum_{i=0}^{s-1}(b^D)^{i+1}(-a)^ia^\pi \sum_{i=0}^{s-1}(b^D)^{i+1}(-a)^ia^\pi\\
&=&[bb^Da^\pi +\sum_{i=1}^{s-1}(-b^Da)^ia^\pi]\sum_{i=0}^{s-1}(b^D)^{i+1}(-a)^ia^\pi +a\sum_{i=0}^{s-1}(b^D)^{i+1}(-a)^ia^\pi \sum_{i=0}^{s-1}(b^D)^{i+1}(-a)^ia^\pi\\
&=&bb^D\sum_{i=0}^{s-1}(b^D)^{i+1}(-a)^ia^\pi-bb^Daa^D\sum_{i=0}^{s-1}(b^D)^{i+1}(-a)^ia^\pi
+\sum_{i=1}^{s-1}(-b^Da)^ia^\pi\sum_{i=0}^{s-1}(b^D)^{i+1}(-a)^ia^\pi\\
&&+ a\sum_{i=0}^{s-1}(b^D)^{i+1}(-a)^ia^\pi \sum_{i=0}^{s-1}(b^D)^{i+1}(-a)^ia^\pi\\
&=& \sum_{i=0}^{s-1}(b^D)^{i+1}(-a)^ia^\pi-bb^Daa^D\sum_{i=0}^{s-1}(b^D)^{i+1}(-a)^ia^\pi
+\sum_{i=1}^{s-1}(-b^Da)^ia^\pi\sum_{i=0}^{s-1}(b^D)^{i+1}(-a)^ia^\pi\\
&&+ a\sum_{i=0}^{s-1}(b^D)^{i+1}(-a)^ia^\pi \sum_{i=0}^{s-1}(b^D)^{i+1}(-a)^ia^\pi\\
&=&  \sum_{i=0}^{s-1}(b^D)^{i+1}(-a)^ia^\pi+bb^D\sum_{i=0}^{s-1}(-ab^D)^{i+1}a^Da^\pi
+\sum_{i=1}^{s-1}(-b^Da)^ia^\pi\sum_{i=0}^{s-1}(b^D)^{i+1}(-a)^ia^\pi\\
&&+ a\sum_{i=0}^{s-1}(b^D)^{i+1}(-a)^ia^\pi \sum_{i=0}^{s-1}(b^D)^{i+1}(-a)^ia^\pi\\
&=& \sum_{i=0}^{s-1}(b^D)^{i+1}(-a)^ia^\pi+\sum_{i=1}^{s-1}(-b^Da)^ia^\pi\sum_{i=0}^{s-1}(b^D)^{i+1}(-a)^ia^\pi\\
&&+a\sum_{i=0}^{s-1}(b^D)^{i+1}(-a)^ia^\pi \sum_{i=0}^{s-1}(b^D)^{i+1}(-a)^ia^\pi\\
&=& \sum_{i=0}^{s-1}(b^D)^{i+1}(-a)^ia^\pi +m_1+m_2,
\end{eqnarray*}
where $$m_1=\sum_{i=1}^{s-1}(-b^Da)^ia^\pi\sum_{i=0}^{s-1}(b^D)^{i+1}(-a)^ia^\pi,$$
$$m_2=a\sum_{i=0}^{s-1}(b^D)^{i+1}(-a)^ia^\pi \sum_{i=0}^{s-1}(b^D)^{i+1}(-a)^ia^\pi.$$

In view of the equality $(3.3)$, it is enough to prove that
\begin{eqnarray*}
m_1+m_2&=&b^\pi a \sum_{i=0}^{s-2}(i+1)(b^D)^{i+2}(-a)^ia^\pi.
\end{eqnarray*}

Since
\begin{eqnarray*}
&&b^\pi a \sum_{i=0}^{s-2}(i+1)(b^D)^{i+2}(-a)^ia^\pi = (1-bb^D)a \sum_{i=0}^{s-2}(i+1)(b^D)^{i+2}(-a)^ia^\pi\\
&=&a \sum_{i=0}^{s-2}(i+1)(b^D)^{i+2}(-a)^ia^\pi-bb^Da \sum_{i=0}^{s-2}(i+1)(b^D)^{i+2}(-a)^ia^\pi,
\end{eqnarray*}
we only need to show $$m_1=-bb^Da \sum_{i=0}^{s-2}(i+1)(b^D)^{i+2}(-a)^ia^\pi,$$
$$m_2=a \sum_{i=0}^{s-2}(i+1)(b^D)^{i+2}(-a)^ia^\pi.$$

Since $a^sa^\pi=0$ and $bb^D$ commutes with $b^Da$, we get
\begin{eqnarray*}
m_1&=&\sum_{i=1}^{s-1}(-b^Da)^ia^\pi\sum_{i=0}^{s-1}(b^D)^{i+1}(-a)^ia^\pi= \sum_{i=1}^s(-b^Da)^ia^\pi\sum_{i=0}^{s-1}(b^D)^{i+1}(-a)^ia^\pi\\
&\stackrel{(2.6)}{=}& \sum_{i=1}^s(-bb^Db^Da)^ia^\pi\sum_{i=0}^{s-1}(-b^Da)^ib^Da^\pi=bb^D\sum_{i=1}^s(-b^Da)^ia^\pi\sum_{i=0}^{s-1}(-b^Da)^ib^Da^\pi\\
&=& -bb^D\sum_{i=1}^s(-b^Da)^{i-1}b^Daa^\pi\sum_{i=0}^{s-1}(-b^Da)^ib^Da^\pi\\
&\stackrel{(2.6)}{=}&-bb^D\sum_{i=0}^{s-1}(-b^Da)^i\sum_{i=0}^{s-1}(-b^Da)^i(b^Daa^\pi)b^Da^\pi\\
&=&-bb^D\sum_{i=0}^{s-1}(-b^Da)^i\sum_{i=0}^{s-1}(-b^Da)^ib^D(b^Daa^\pi)a^\pi\\
&=&-bb^D\sum_{i=0}^{s-1}(-b^Da)^i\sum_{i=0}^{s-1}(-b^Da)^i(b^D)^2aa^\pi\\
&\stackrel{(2.6)}{=}& -bb^D(b^D)^2a\sum_{i=0}^{s-1}(-b^Da)^i\sum_{i=0}^{s-1}(-b^Da)^ia^\pi\\
&\stackrel{(2.6)}{=} &-b(b^D)^2ab^D\sum_{i=0}^{s-1}(-b^Da)^i\sum_{i=0}^{s-1}(-b^Da)^ia^\pi\\
&\stackrel{(2.6)}{=}&-bb^Da(b^D)^2\sum_{i=0}^{s-1}(-b^Da)^i\sum_{i=0}^{s-1}(-b^Da)^ia^\pi \\
&=&-bb^Da(b^D)^2\sum_{i=0}^{s-2}(i+1)(-b^Da)^ia^\pi\\
&\stackrel{(2.13)}{=}& -bb^Da\sum_{i=0}^{s-2}(i+1)(b^D)^{i+2}(-a)^ia^\pi.
\end{eqnarray*}

Similarly,
\begin{eqnarray*}
m_2&=&a\sum_{i=0}^{s-1}(b^D)^{i+1}(-a)^ia^\pi \sum_{i=0}^{s-1}(b^D)^{i+1}(-a)^ia^\pi \stackrel{(2.6)}{=} a\sum_{i=0}^{s-1}(-b^Da)^ib^Da^\pi \sum_{i=0}^{s-1}(-b^Da)^ib^Da^\pi\\
&\stackrel{(2.6)}{=}&a\sum_{i=0}^{s-1}(-b^Da)^i\sum_{i=0}^{s-1}(-b^Da)^ib^Da^\pi b^Da^\pi\stackrel{(2.6)}{=}a\sum_{i=0}^{s-1}(-b^Da)^i\sum_{i=0}^{s-1}(-b^Da)^i b^Db^Da^\pi a^\pi\\
&=&a\sum_{i=0}^{s-1}(-b^Da)^i\sum_{i=0}^{s-1}(-b^Da)^i (b^D)^2 a^\pi\stackrel{(2.6)}{=}a(b^D)^2\sum_{i=0}^{s-1}(-b^Da)^i\sum_{i=0}^{s-1}(-b^Da)^i  a^\pi\\
&=& a(b^D)^2\sum_{i=0}^{s-2}(i+1)(-b^Da)^i  a^\pi\stackrel{(2.13)}{=} a(b^D)^2\sum_{i=0}^{s-2}(i+1)(b^D)^i(-a)^i a^\pi\\
&=&  a\sum_{i=0}^{s-2}(i+1)(b^D)^{i+2}(-a)^i a^\pi.
\end{eqnarray*}

Thus,
\begin{eqnarray*}
z_3&=& \sum _{i=0}^{s-1}(b^D)^{i+1}(-a)^ia^\pi+b^\pi a \sum_{i=0}^{s-2}(i+1)(b^D)^{i+2}(-a)^ia^\pi.
\end{eqnarray*}

Therefore, $x(a+b)x=x$.

Step 3. We prove that $(a+b)-(a+b)^2x$ is nilpotent.

Note that the proof of step 1. We have
\begin{eqnarray*}
&&(a+b)-(a+b)^2x = (a+b)-[a^D\xi^D+\sum_{i=0}^{s-1}(b^D)^{i+1}(-a)^ia^\pi](a+b)^2\\
&\stackrel{(2.5)}{=}&(a+b)-\xi^Da^D(a+b)^2-\sum_{i=0}^{s-1}(b^D)^{i+1}(-a)^i a^\pi(a+b)^2 \\
&=&(a+b)-\xi^Da(a^D(a+b))^2-\sum_{i=0}^{s-1}(b^D)^{i+1}(-a)^i a^\pi(a^2+ab+ba+b^2)\\
&=& a\xi+a^\pi b-\xi^Da(\xi-a^\pi)^2-\sum_{i=0}^{s-1}(b^D)^{i+1}(-a)^i a^\pi a^2\\
&&-\sum_{i=0}^{s-1}(b^D)^{i+1}(-a)^i a^\pi ab-\sum_{i=0}^{s-1}(b^D)^{i+1}(-a)^i a^\pi ba\\
&&- \sum_{i=0}^{s-1}(b^D)^{i+1}(-a)^i a^\pi b^2\\
&\stackrel{(2.13)}{=}& a\xi+a^\pi b-\xi^Da(\xi^2-a^\pi)+\sum_{i=0}^{s-1}(-b^Da)^{i+1}a a^\pi+\sum_{i=0}^{s-1}(-b^Da)^{i+1}a^\pi b \\
&&-\sum_{i=0}^{s-1}(-b^Da)^ib^Da^\pi ba - \sum_{i=0}^{s-1}b^D(-b^Da)^ia^\pi b^2\\
&\stackrel{(3.2)}{=}& a\xi+a^\pi b-\xi^Da(\xi^2-a^\pi)+\sum_{i=0}^{s-1}(-b^Da)^{i+1}a a^\pi+\sum_{i=0}^{s-1}(-b^Da)^{i+1}a^\pi b \\
&&-\sum_{i=0}^{s-1}(-b^Da)^ibb^Daa^\pi - \sum_{i=0}^{s-1}b^D(-b^Da)^ia^\pi b^2\\
&=& a\xi+a^\pi b-\xi^Da(\xi^2-a^\pi)-bb^Daa^\pi +\sum_{i=0}^{s-1}(-b^Da)^{i+1}a^\pi b \\
&&- \sum_{i=0}^{s-1}b^D(-b^Da)^ia^\pi b^2\\
&\stackrel{(2.6)}=& a\xi+a^\pi b-\xi^Da(\xi^2-a^\pi)-bb^Daa^\pi +\sum_{i=0}^{s-1}(-b^Da)^{i+1}a^\pi b \\
&&- \sum_{i=0}^{s-1}(-b^Da)^ib^Da^\pi b^2\\
&\stackrel{(3.2)}{=}& a\xi+a^\pi b-\xi^Da(\xi^2-a^\pi)-bb^Daa^\pi +\sum_{i=0}^{s-1}(-b^Da)^{i+1}a^\pi b \\
&&- \sum_{i=0}^{s-1}(-b^Da)^ibb^Da^\pi b\\
&\stackrel{(2.6)}{=}& a\xi+a^\pi b-\xi^Da(\xi^2-a^\pi)-bb^Daa^\pi +\sum_{i=0}^{s-1}(-b^Da)^{i+1}a^\pi b \\
&&- \sum_{i=0}^{s-1}bb^D(-b^Da)^ia^\pi b\\
&=&a\xi +a^\pi b-\xi^Da\xi^2+\xi^Daa^\pi -bb^Daa^\pi-bb^Da^\pi b\\
&=&(a\xi-\xi^Da\xi^2)+\xi^Daa^\pi+(a^\pi b-bb^Da^\pi b)-bb^Daa^\pi\\
&=&  a\xi \xi^\pi+\xi^Daa^\pi +b^\pi a^\pi b -bb^Daa^\pi\\
&=& b_1+b_2.
\end{eqnarray*}

By Lemma 2.8, $(a+b)-(a+b)^2x=b_1+b_2$ is nilpotent.

The proof is completed.
\end{proof}

\begin{corollary} $[16, {\rm Theorem}~3]$ Let $a,b \in R^D$ with $ab=ba$. Then $a+b$ is Drazin invertible if and only if $1+a^Db$ is Drazin invertible. In this case, we have
\begin{eqnarray*}
(a+b)^D &=& (1+a^Db)^D a^D+b^D(1+aa^\pi b^D)^{-1}a^\pi\\
&=& a^D(1+a^Db)^Dbb^D+b^\pi(1+bb^\pi a^D)^{-1}+b^D(1+aa^\pi b^D)^{-1}a^\pi,
\end{eqnarray*}
and
\begin{center}
$(1+a^Db)^D=a^\pi +a^2a^D(a+b)^D$.
\end{center}
\end{corollary}
\begin{proof} Since $ab=ba$, we have $a$, $b$, $a^D$ and $b^D$ commute with each other. Hence, it follows that $a^\pi b[a^D(1+a^Db)^D]^2=0$ and $b^\pi a \displaystyle{\sum_{i=0}^{s-2}(i+1)}(b^D)^{i+2}(-a)^ia^\pi=0$, where $s={\rm ind}(a)$.

Since $aa^\pi b^D$ is nilpotent, $1+aa^\pi b^D$ is invertible and
\begin{eqnarray*}
(1+aa^\pi b^D)^{-1} &=& 1+(-aa^\pi b^D)+(-aa^\pi b^D)^2+\cdots +(-a a^\pi b^D)^{s-1}\\
&=& \sum _{i=0}^{s-1}(-b^Daa^\pi)^i.
\end{eqnarray*}
We have
\begin{eqnarray*}
b^D(1+aa^\pi b^D)^{-1}a^\pi &=& b^D \sum _{i=0}^{s-1}(-b^Daa^\pi)^i a^\pi= b^D \sum_{i=0}^{s-1}(b^D)^i(-a)^ia^\pi\\
&=&\sum_{i=0}^{s-1}(b^D)^{i+1}(-a)^ia^\pi.
\end{eqnarray*}

Therefore, $(a+b)^D = (1+a^Db)^D a^D+b^D(1+aa^\pi b^D)^{-1}a^\pi$.
 \end{proof}

\centerline {\bf ACKNOWLEDGMENTS} This research is supported by the National Natural Science Foundation of China (10971024),
the Specialized Research Fund for the Doctoral Program of Higher Education (20120092110020), the Natural Science Foundation of Jiangsu Province (BK2010393) and the Foundation of Graduate Innovation Program of Jiangsu Province(CXLX13-072).

\end{document}